\documentclass[11pt,preprint]{elsarticle}
\usepackage{amssymb, amstext, amscd, amsmath, color}

\usepackage{url}
\usepackage{amsfonts}
\usepackage{amsthm}
\usepackage{amsfonts}
\usepackage{array}

\usepackage{epsfig}
\usepackage{eucal}
\usepackage{latexsym}
\usepackage{mathrsfs}
\usepackage{textcomp}
\usepackage{verbatim}
\usepackage{setspace}
\usepackage{subfigure}

\newtheorem{thm}{Theorem}[section]

\newtheorem{con}[thm]{Conjecture}
\newtheorem{cor}[thm]{Corollary}

\newtheorem{defn}[thm]{Definition}

\newtheorem{lem}[thm]{Lemma}

\newtheorem{prop}[thm]{Proposition}
\newtheorem{rem}[thm]{Remark}

\numberwithin{equation}{section}

\newcommand{\bN}{{\mathbb{N}}}

\newcommand{\bR}{{\mathbb{R}}}

  \newcommand{\M}{{\mathcal{M}}}

  \newcommand{\R}{{\mathcal{R}}}

  \newcommand{\X}{{\mathcal{X}}}
  \newcommand{\Y}{{\mathcal{Y}}}
  
\begin{document}

\begin{frontmatter}

\title{A Constructive Characterisation of Circuits in the Simple $(2,2)$-sparsity Matroid}
\author{Anthony Nixon}
\address{Heilbronn Institute for Mathematical Research, Department of Mathematics, University of Bristol, BS8 1TW, tony.nixon@bristol.ac.uk}

\begin{abstract}
We provide a constructive characterisation of circuits in the simple $(2,2)$-sparsity matroid. A circuit is a simple graph $G=(V,E)$
with $|E|=2|V|-1$ and the number of edges induced by any $X \subsetneq V$ is at most $2|X|-2$.
Insisting on simplicity results in the Henneberg 2 operation being enough only when the graph is sufficiently connected. Thus we
introduce $3$ different join operations to complete the characterisation.
Extensions are discussed to when the sparsity matroid is connected and this 
is applied to the theory of frameworks on surfaces to provide a conjectured characterisation of when frameworks on an infinite circular 
cylinder are generically globally rigid.
\end{abstract}


\end{frontmatter}

\section{Introduction}

For $k,l \in \bN$ a multigraph $G=(V,E)$ is $(k,l)$-tight if $|E|=k|V|-l$ and for every subgraph $G'=(V',E')$ the inequality $|E'|\leq k|V'|-l$ holds. 
It is well known that the edge sets of such multigraphs induce matroids when $l<2k$ \cite{L&S}, \cite{Whi3}; we denote these matroids as 
$M(k,l)$. These multigraphs can be decomposed into unions of trees and map graphs \cite{N-W}, \cite{Tut}, \cite{Whi5}; correspondingly the 
matroids are unions of cycle and bicycle matroids. There is an elegant recursive construction of the bases in $M(k,l)$ due to Fekete and 
Szeg\"{o} \cite{F&Z}. Their result is built on the construction of Tay \cite{Tay2} for $k=l$. In this case a recursive characterisation of 
circuits in $M(k,k)$ can be found as 
a special case of a theorem of Frank and Szeg\"{o} \cite{F&S2} on highly $k$-tree connected multigraphs. These characterisations use 
generalisations of the Henneberg moves \cite{Hen}. However each list of construction moves is insufficient if we restrict to (simple) graphs at each stage of the induction. 

The rigidity of frameworks on surfaces is a situation with exactly this constraint. The relevant graphs are the $(2,l)$-tight  graphs \cite{NOP}, \cite{N&O}. When the $(k,l)$-tight graph is simple, they
still induce a matroid and we denote it as $M^*(k,l)$. Recursive constructions for the bases of $M^*(2,l)$ ($l=2,1$) can be found in
\cite{NOP} and \cite{N&O}. It should be noted, however, that these require more than just Henneberg moves. Throughout we will call a vertex of degree 3 a \emph{node}.
The \emph{Henneberg $2$ move} adds a node to a graph by subdividing an edge and connecting the new vertex to a third existing vertex. Other Henneberg moves will not be relevant here.

In this paper we prove a 
constructive characterisation of circuits in $M^*(2,2)$. The corresponding result
for circuits in $M^*(2,3)=M(2,3)$ was proved by Berg and Jord\'{a}n \cite{B&J}. A circuit in $M^*(2,l)$ $(l=2,3$) necessarily contains a node. However there may be no node that is suitable for an inverse Henneberg 2 operation.
This is the key difficulty in extending constructive characterisations from bases to circuits.

Berg and Jord\'{a}n \cite{B&J} showed that a circuit in $M^*(2,3)$ has a suitable node whenever the graph 
is $3$-connected (in the vertex sense). Thus the combination of the Henneberg $2$ operation and the 2-sum operation \cite{Oxl} which glues two circuits together over a $2$-vertex cut were sufficient to generate all such circuits. Correspondingly our main results are analogues for circuits in $M^*(2,2)$. We require both additional conditions to ensure the (inverse of the) Henneberg 2 operation can be used and more elaborate joining techniques than the 2-sum. 

From here on we define a \emph{circuit} (resp. \emph{multicircuit}) to be the graph (resp. multigraph) induced by a circuit in $M^*(2,2)$ (resp. $M(2,2)$) i.e. a 
graph (resp. multigraph) $G=(V,E)$ with $|E|=2|V|-1$ and for every proper subgraph 
$H=(V',E')\subset G$ we have $|E'|\leq 2|V'|-2$.
Figure \ref{basegs} gives three small examples of circuits.

Let $K_4\sqcup K_4$ denote the unique graph formed by two copies of $K_4$ intersecting in a single edge and let $K_4 \veebar K_4$
denote the unique graph formed from two copies of $K_4$ intersecting in a single vertex by adding any edge.
We will say that $K_5\setminus e$, $K_4 \sqcup K_4$ and $K_4\veebar K_4$ are \emph{base graphs}, see Figure \ref{basegs}.

\begin{center}
\begin{figure}[ht]
\centering
\includegraphics[width=10.4cm]{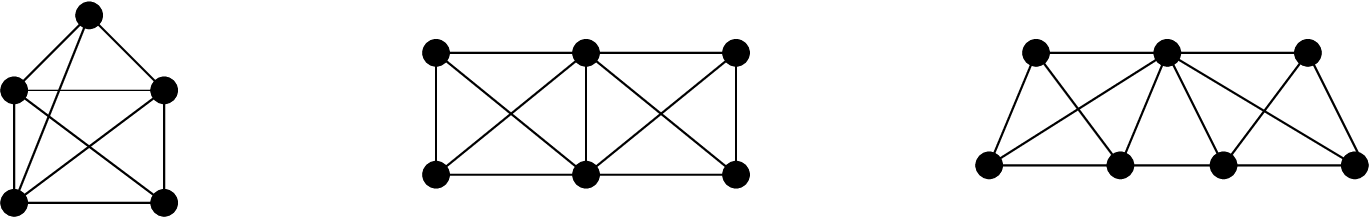}
\caption{From left to right: $K_5\setminus e$, $K_4 \sqcup K_4$ and $K_4\veebar K_4$.}
\label{basegs}
\end{figure}
\end{center}

\vspace{-1cm}

For the inverse Henneberg 2 operation, let $G=(V,E)$ be a graph and let $G_v^{uw}$ denote the graph formed 
by removing a node $v$ from $G$ and adding the edge $uw$ where $u,w \in N(v)$ (the neighbour set
of $v$).
Let $G$ be a circuit and let $v$ be a node in $G$.
The pair of edges $uv,wv$ is \emph{admissible} if $G_v^{uw}$ is a circuit.
A node $v$ is \emph{admissible} if there is $u,w \in N(v)$ such that $uv,wv$ is admissible. Figures \ref{fig:example} illustrates admissibility.

\begin{center}
\begin{figure}[ht]
\centering
\includegraphics[width=7cm]{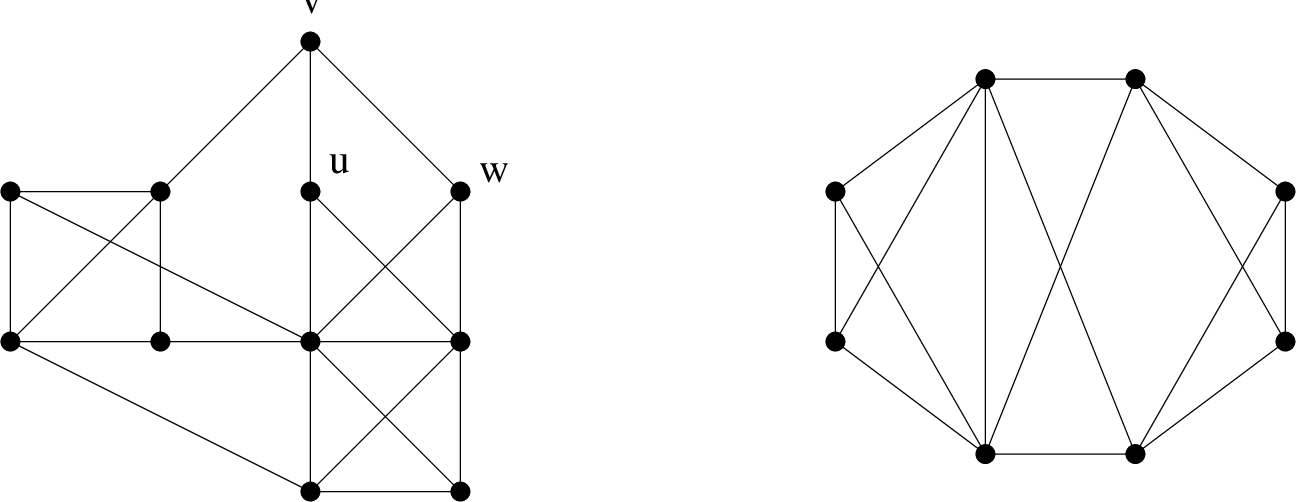}
\caption{$v$ is a non-admissible node in a $3$-connected circuit with no non-trivial 3 edge-cutsets. Choosing $uw$ as the new edge creates a copy of $K_4 \sqcup K_4$ and not choosing $uw$ leaves a degree 2 vertex. $u$ and $w$ are examples of admissible nodes. The second circuit contains no admissible nodes.}
\label{fig:example}
\end{figure}
\end{center}

\vspace{-1cm}

By a \emph{non-trivial $k$-edge cutset} we mean a $k$-edge-cutset in which the two components have at least two vertices. Since every 
circuit contains a degree $3$ vertex, there always exist trivial $3$-edge-cutsets. Since we will primarily be considering non-trivial 
$3$-edge cutsets in $3$-connected graphs we may assume the edges in any such cutset are disjoint.

\begin{thm}\label{recursivethm2}
Let $G$ be a $3$-connected circuit in $M^*(2,2)$ with no non-trivial $3$-edge cutsets and $|V| \geq 6$. Then $G$ has two admissible nodes.
\end{thm}

\begin{center}
\begin{figure}[ht]
\centering
\includegraphics[width=9cm]{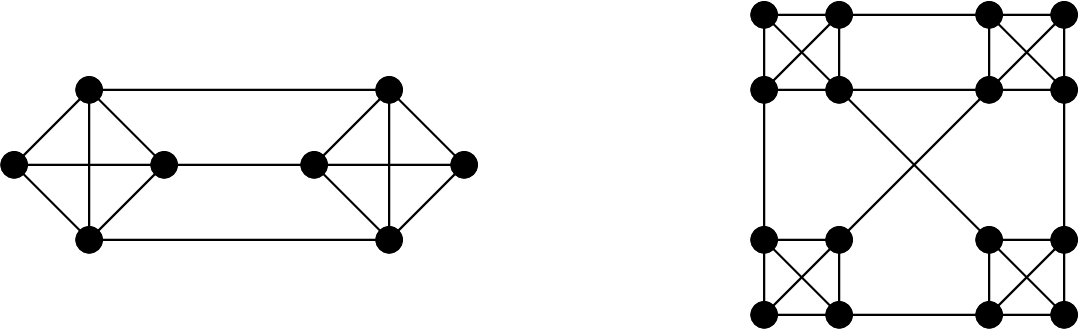}
\caption{Two $3$-connected circuits with no admissible nodes.}
\label{novk4}
\end{figure}
\end{center}

\vspace{-1cm}

In Lemma \ref{admlem9} we see that circuits are 2-connected and 3-edge connected. Indeed Figure \ref{novk4} illustrates circuits not covered by this result. Thus
to extend Theorem \ref{recursivethm2} to cover all circuits in $M^*(2,2)$ we introduce the following 3 operations.
Let $G_1,G_2$ be circuits such that $G_1$ contains an edge $a_1b_1$ and $G_2$ contains a two vertex cut
$a_2,b_2$ within $K_4(a_2,b_2,c_2,d_2)$. A \emph{$1$-join operation} takes $G_1$ and $G_2$ and forms $G_1 \oplus_1 G_2$ by removing $a_1b_1$,
$c_2,d_2$ and $a_2b_2$ and superimposing $a_1,b_1$ onto $a_2,b_2$ and calling the resulting vertices $a,b$.
Secondly, let $G_1,G_2$ be circuits such that $G_i$ contains a two vertex cut $a_i,b_i$ with one component inducing $K_4(a_i,b_i,c_i,d_i)$. A 
\emph{$2$-join operation} takes $G_1$ and $G_2$ and forms $G_1 \oplus_2 G_2$ by removing $c_i,d_i$ and superimposing $a_1,b_1$ onto 
$a_2,b_2$ and calling the resulting vertices $a,b$ and keeping only one copy of the edge $ab$.
Finally, let $G_1,G_2$ be circuits such that $G_i$ contains a node $v_i$ with $N(v_i)=\{a_i,b_i,c_i\}$. A \emph{$3$-join 
operation} takes $G_1$ and $G_2$ and forms $G_1 \oplus_3 G_2$ by deleting $v_1,v_2$ and adding edges $a_1a_2,b_1b_2,c_1c_2$. See Figure \ref{fig:sums}.

\begin{thm}\label{recursivethm1}
A graph $G$ is a circuit in $M^*(2,2)$ if and only if $G$ can be generated recursively from disjoint copies of base graphs by applying Henneberg $2$ moves
within connected components and taking $1$-joins, $2$-joins or $3$-joins of different connected components.
\end{thm}

\begin{center}
\begin{figure}[ht]
\centering
\subfigure{\includegraphics[width=7.5cm]{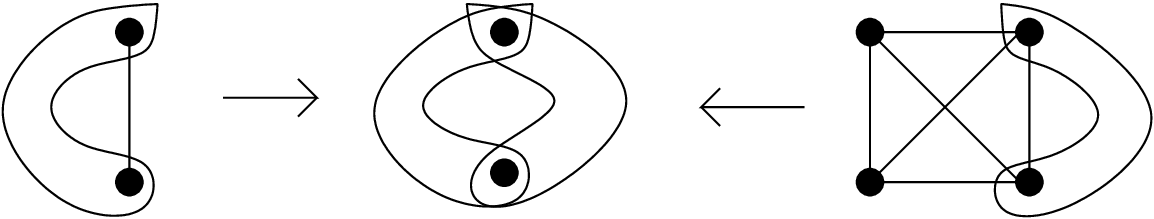}}\vspace{0.4cm}
\subfigure{\includegraphics[width=7.5cm]{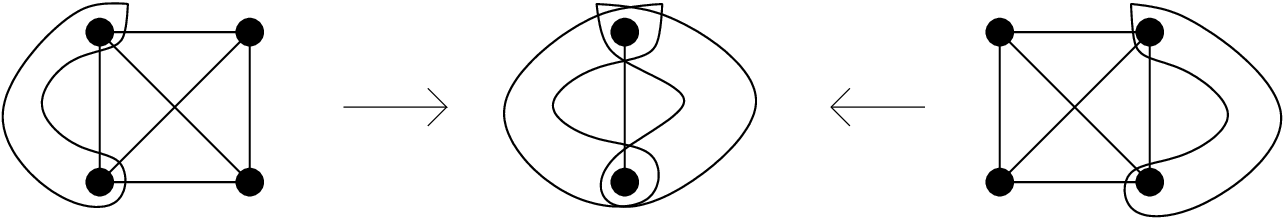}}\vspace{0.4cm}
\subfigure{\includegraphics[width=7.5cm]{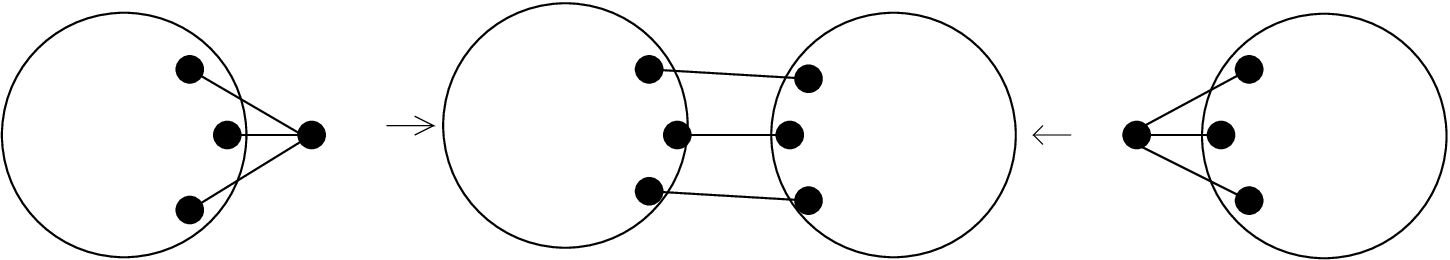}}
\caption{The $1$-, $2$- and $3$-join operations.}
\label{fig:sums}
\end{figure}
\end{center}

\vspace{-1cm}

\subsection{Outline}

In Section \ref{ansec} we prove Theorem \ref{recursivethm2}. We start with some elementary properties of circuits culminating in Lemma \ref{admlem1} where we establish two blocks to admissibility: (a) preserving simplicity and (b) preserving subgraph sparsity.
The key novelty in Section \ref{ansec} is in dealing with (a). Proposition \ref{prop:k4fix} establishes the level of connectivity required to guarantee nodes not contained in copies of $K_4$. Combining this with Lemma \ref{star} largely allows us to reduce to (b), which is considered in Subsection \ref{subsec:guarantee}. This follows the method of \cite{B&J} establishing structural results for circuits with non-admissible nodes. 
The proof of Theorem \ref{recursivethm2} is completed by deducing from Proposition \ref{prop:k4fix} that a special subforest of nodes is non-empty and combining this with these structural results. 

In Section
\ref{sumsec} we analyse the $1$-, $2$- and $3$-join operations which extend standard matroid sum techniques. There is one final technical point needed. In Section \ref{sec:recursion} we translate a circuit in $M^*(2,2)$ into a circuit in $M(2,2)$ in order to establish admissibility in circuits which are not sufficiently connected but have special structure. We then combine the results up to this point to prove Theorem \ref{recursivethm1}. 

In Section \ref{sec:rigidity} we consider connectedness in $M^*(2,2)$ and obtain a precise analogue of \cite[Theorem $3.2$]{J&J}. This is used to link our results to the unique realisation problem for frameworks in $3$-dimensions supported on an infinite circular cylinder. We finish by conjecturing a combinatorial description of when such a realisation is unique, Conjecture \ref{cylinderglobal}, and outlining some extensions.

\subsection{Comparing Constructions}

In \cite{F&S2} it was shown that all circuits in $M(2,2)$ can be generated from a single loop using Henneberg 2 operations. Figures 
\ref{basegs} and \ref{novk4} give examples of graphs for which this construction requires multigraphs in the intermediate steps. Moreover repeated application of, say, $3$-join operations on these examples give arbitrarily large circuits with no admissible nodes.

Since every node in each of these examples is contained in a copy of $K_4$ it would be natural to consider a recursive operation in which a copy of $K_4$ was contracted to a single vertex.
In \cite{N&O} a recursive construction was given for bases in $M^*(2,1)$ which included exactly such an operation. 
While it is possible that there is a construction that utilises this type of operation it is not clear what the construction should be. The operation used in \cite{N&O}, expanding a vertex into a copy of $K_4$, need not preserve the circuit property (see Lemma \ref{admlem9}) and the inverse operation need not preserve simplicity. In \cite{N&O} this led to the use of vertex splitting (the inverse of contracting an edge in a triangle). However this also has similar issues for circuits. One advantage of our construction lies in the fact that the $1$-, $2$- and $3$-join moves and their inverses preserve the circuit property.

\subsection{Preliminaries}

In this paper graphs have no loops or multiple edges, while multigraphs may have both. If $G=(V,E)$ is a graph with $v \in V$ then 
$d_G(v)$ denotes the degree of $v$ in $G$ and $N(v)$ denotes the neighbour set of $v$.

Let $f(H)=2|V'|-|E'|$.
For $X\subset V$ we let $i_G(X)$ denote the number of edges in the subgraph of $G$ induced by $X$. We drop the subscript
when the graph is clear from the context. If $X$ and $Y$ are disjoint subsets of the vertex set $V$ of a given graph $G$, then we use 
$d(X,Y)$ to denote the number of edges from $X$ to $Y$ and $d(X):=d(X,V\setminus X)$. 

Let $X_1,X_2$ be subsets of $V$ and $X_1',X_2'$ their induced subgraphs. 
Then $i(X_1\cup X_2)=i(X_1)+i(X_2)-i(X_1\cap X_2)+d(X_1,X_2)$ whereas $f(X_1'\cup X_2')=f(X_1')+f(X_2')-f(X_1'\cap X_2')$. Hence when 
$d(X_1,X_2)>0$ these counting notions will have different results.

\section{Admissible Nodes}
\label{ansec}

In this section we establish when the Henneberg 2 operation and its inverse preserve the circuit property. The first part of this is 
elementary and we omit the proof.

\begin{lem}\label{hen2move}
 Let $G'$ be formed from $G$ by a Henneberg $2$ move and let $G$ be a circuit. Then $G'$ is a circuit.
\end{lem}

Under what conditions there is an inverse Henneberg 2 move for which the corresponding statement holds is the subject of this section.

\subsection{Basic Properties of Circuits}

Let $G=(V,E)$.
We say that a subset $X\subset V$ is \emph{critical} if $i(X)=2|X|-2$.
The following is a simple analogue of \cite[Lemma $2.3$]{B&J} and we omit the proof.

\begin{lem}\label{admlem2}
 Let $G=(V,E)$ be a circuit and let $X, Y \subset V$ be critical such that $|X \cap Y| \geq 1$ and 
$|X \cup Y| \leq |V|-1$. Then $X\cap Y$ and $X \cup Y$ are both critical, and 
$d(X \setminus Y, Y \setminus X)=0$.
\end{lem}

Let $G=(V,E)$ be a circuit. For any critical set $X \subset V$, $G[X]$ is connected but need not be $2$-connected.

\begin{lem}\label{admlem9}
Let $G$ be a circuit. Then $G$ is $2$-connected and $3$-edge-connected.
\end{lem}

\begin{proof}
Let $G=(V,E)$. Suppose there exists $v \in V$ such that $G \setminus v$ has a bipartition $A,B$ with no edges from 
$A$ to $B$.
\begin{eqnarray*}
2|V|-1=|E| &=& |E(A \cup v)| +|E(B \cup v)| \\ & \leq & 2(|A| +1)-2 + 2(|B|+1)-2 \\ 
&=& 2|V|-2, 
\end{eqnarray*}
a contradiction.
This proves the first statement, the second is similar.
\end{proof}

The following is easy and similar to \cite[Lemma $2.5$]{B&J}. We omit the proof.

\begin{lem}\label{admlem5}
 Let $G=(V,E)$ be a circuit. Let $X \subset V$ be a critical set. 
Then $V \setminus X$ contains at least one node (in $G$).
\end{lem}

Our next lemma gives a criterion for admissibility.

\begin{lem}\label{admlem1}
 Let $G$ be a circuit, let $v$ be a node in $G$ with $N(v)=\{u,w,z\}$. Then $uv,wv$ is not 
admissible if and only if either (a) $uw \in E$ or (b) there is a critical set $X \subset V$ with $u,w \in X$ and $v,z \notin X$.
\end{lem}

\begin{proof}
Suppose first that (b) holds. Then the inverse Henneberg
$2$ move creates a new edge $uw$ implying $i(X)=2|X|-1$ and $X \subsetneq V$.
Also if (a) holds then $G_v^{uw}$ is not a simple graph.

Conversely, if $uv,wv$ is not admissible and (a) fails there is $X \subset V(G_v^{uw})$ such that $G[X]$ is not $(2,2)$-sparse. Then 
$|E(X)| \geq 2|X|-1$. It follows that $X$ is critical in $G$ and $u,w \in X$. If $z \in X$ then
$|E(X \cup v)|=|E(X)|+3 = 2|X|-2+3= 2|X \cup v|-1$, 
a contradiction. Thus $z \notin X$.
\end{proof}

Condition (b) in Lemma \ref{admlem1} leads us to strengthen the definition of critical as follows.
Let $G=(V,E)$ be a circuit. For a node $v \in V$ with $N(v)=\{u,w,z \}$ we say that a critical 
set $X$ is \emph{$v$-critical} if $u,w \in X$ and $v,z \notin X$. 
If $z$ is a node and such an $X$ exists then an inverse Henneberg $2$ move on $uv,wv$ is not admissible. Here 
$V \setminus \{v,z\}$ is a \emph{trivial} $v$-critical set on $u$ and $w$. If $X$ is a $v$-critical set on $u$ and $w$ for some node $v$ with $N(v)=\{u,w,z\}$ and $d_G(z) \geq 4$ then $X$ is \emph{node-critical}.  We will return to node-critical sets in Subsection \ref{subsec:guarantee}.

\subsection{Preserving Simplicity}

Condition (a) in Lemma \ref{admlem1} is crucial in separating the problem at hand from the analogue in \cite{B&J}. 
The following Proposition is the key step in bridging this difficulty.

\begin{prop}\label{prop:k4fix}
Let $G=(V,E)$ be a $3$-connected circuit with no non-trivial $3$-edge-cutsets and $|V|\geq 6$. 
Let $X_1,\dots, X_n$ be critical sets and let $Y=V\setminus \bigcup_{i=1}^nX_i$. Suppose
\begin{enumerate}
\item $|Y|\geq 2$,
\item $\bigcup_{i=1}^m G[X_i]$ is disconnected, or
\item $X_1,\dots, X_n$ induce copies of $K_4$.
\end{enumerate}
Then $Y$ contains at least two nodes of $G$.
\end{prop}

\begin{proof}
We prove $1$ and $2$ simultaneously.
With vertices labelled $v_1,\dots,v_{|V|}$, since $|E|=2|V|-1$ we have 
\[ \sum_{i=1}^{|V|}(4-d_G(v_i))=2.\]
Let $Z_1,\dots,Z_m$ be the connected components in $\bigcup_{i=1}^m G[X_i]$. 
In cases 1 and 2 Lemma \ref{admlem2} implies  $X_i \cup X_j$ is critical and $d(X_i,X_j)=0$ or $X_i \cap X_j=\emptyset$ for each 
$1 \leq i<j\leq n$.
%
Now $i(Z_j)=2|Z_j|-2$ for each $j$. Thus
\[ \sum_{i=1}^{|Z_j|}(4-d_{G[Z_j]}(v_i))=4.\]
By assumption 1 or 2 $|V\setminus Z_j|\geq 2$ so there are at least $4$ edges of the form $xy$ with $x \in Z_j$, 
$y \in V\setminus Z_j$. This implies 
\[ \sum_{i=1}^{|Z_j|}(4-d_G(u_i))\leq 0\] 
(with the vertices in $Z_j$ labelled $u_1,\dots,u_{|Z_j|}$) for each $j$. Thus
\[ \sum_{j=1}^m(\sum_{i=1}^{|Z_j|}(4-d_G(u_i)))\leq 0. \]
Since the minimum degree in $G$ is $3$ comparing this with the first summation implies $Y$ contains at least two nodes.

For $3$ assume $X_1,\dots, X_n$ induce copies of $K_4$ and suppose $m=1$ and $|Y|<2$. 
Let $|Y|=1$ then $Z_1$ is critical, $G[Z_1]$ is connected 
and every edge in $G[Z_1]$ is in a copy of $K_4$. 
Since every $A \subsetneq X_i$ with $|A|>1$ satisfies $i(A)\leq 2|A|-3$ we must have $X_1 \cap X_i=a$ for some $i$.
If $a$ is a cut-vertex in $G[Z_1]$ then we guarantee a cutpair in $G$ which contradicts our assumptions so $m>1$.
However if $a$ is not a cut-vertex there is a path in $G[Z_1]$ from any vertex in $X_1\setminus a$ to any vertex in $X_i \setminus a$. 
Since $d(X_1,X_i)=0$ the only way this may happen is if there is a set containing some $y_1 \in X_1\setminus a$ and some 
$y_k \in X_i \setminus a$ which is not contained in $X_1\cup X_i$. Let the path use vertices $y_1,y_2,\dots,y_k$ for some $k\geq 2$
and choose $X'$ to be the union of all $X_j$'s containing some $y_j$ except $X_1$ and $X_i$. Then $X'$ is critical.
As $X_1\cup X_i$ is critical this implies that $i(X'\cup X_1\cup X_i)>2|X'\cup X_1\cup X_i|-2$. Thus $a$ must be a cut-vertex.

A similar argument applies when $Y=0$; here $Z_1=V$ and there is exactly one edge $e$ not in a copy of $K_4$. As above we find $a$ is a 
cut-vertex for $G\setminus e$ and hence a cut-pair exists in $G$.
Therefore $m\geq 2$ and the result follows from $2$.
\end{proof}

Let $V_3=\{v \in V: v \mbox{ is a node}\}$.
Let $V_3^* \subset V_3$ be the subgraph of nodes which are not contained in copies of $K_4$ (in $G$). 
Following \cite{B&J} we call a node $v$ with $d_{G[V_3^*]}(v)\leq 1$ a \emph{leaf node}, with $d_{G[V_3^*]}(v)=2$ a \emph{series node} and with
 $d_{G[V_3^*]}(v)=3$ a \emph{branching node}.
From Proposition \ref{prop:k4fix} we can derive an analogue of \cite[Lemma $2.1$]{B&J}.

\begin{lem}\label{nok4forest}
Let $G=(V,E)$ be a $3$-connected circuit with $|V|\geq 6$ and no non-trivial $3$-edge cutsets. Then $G[V_3^*]$ is a forest on at least two vertices.
\end{lem}

\begin{proof}
By Proposition \ref{prop:k4fix} part $3$ $|V_3^*|\geq 2$.
Suppose $C \subset V_3^*$ induces a cycle. $G$ is not a cycle so $\bar{C} :=V \setminus C \neq \emptyset$. 
$|\bar{C}| >1$ since $G$ is not a wheel. Now
\begin{eqnarray*}
i(\bar{C}) &=& 2|V|-1-i(C)-d(C,\bar{C})= 2|V|-1-|C|-|C|\\ &=& 2(|V|-|C|)-1= 2|\bar{C}|-1, 
\end{eqnarray*}
a contradiction.
\end{proof}

It is much easier to deal with the case when the neighbour set of a node neither induces $K_3$ or induces a graph with no edges.

\begin{lem}\label{star}
Let $G=(V,E)$ be a circuit containing a node $v$ with $N(v)=\{w,u,z\}$. Suppose that either
\begin{enumerate}
\item $G$ is $3$-connected, $uz \notin E$ and $wz,wu \in E$ or
\item $uz, wu \notin E$ and $wz \in E$.
\end{enumerate}
Then $v$ is admissible.  
\end{lem}

\begin{proof}
Since $z,u$ is not a cutpair, $d_G(w)\geq 4$. Let $t \in N(w)$ and suppose $v$ is not admissible. By Lemma \ref{admlem1} there exists
a proper critical subset $X_{zu} \subset V$ containing $z,u$ but not $w,v$. If $t \in X_{zu}$ then $i(X_{zu}\cup w)=2|X_{zu}\cup w|-1$,
a contradiction as $v \notin X_{zu}\cup w$. If $t \notin X_{zu}$ then $X_{zu}\cup w$ is critical and 
$i(X_{zu}\cup w \cup v)=2|X_{zu}\cup w \cup v|-1$, a contradiction as $t \notin X_{zu}\cup w \cup v$. This proves 1.

Now assume for a contradiction that $v$ is not admissible. By Lemma \ref{admlem1} there exists proper 
critical sets $X_{wu},X_{uz} \subset V$. Note $d_G(z)\geq 4$ since $|N(z)\cap X_{uz}|\geq 2$ and similarly $d_G(w)\geq 4$.
By Lemma \ref{admlem2} $X_{wu}\cup X_{uz}$ is critical so adding $wz$ then $v$ plus its 
three edges gives a contradiction. Thus at most one of the critical sets $X_{wu}$ and $X_{uz}$ can exist and 2 follows. 
\end{proof}

\subsection{Guaranteeing an admissible node}
\label{subsec:guarantee}

Now we have reduced the problem to the case when the 3 neighbours of a node induce a null graph. For this we modify results from \cite{B&J}.

\begin{lem}\label{admlem8}
Let $G=(V,E)$ be a circuit with $|V| \geq 6$. Suppose $v$ is a non-admissible node of $G$ with $N(v)=\{x,y,z\}$ and none of $xy,xz,yz$ 
present in $E$. Then there exists two $v$-critical sets $X,Y$ such that $X \cup Y = V \setminus v$. Moreover we may choose $X,Y$ such 
that $z \in X \cap Y$.
\end{lem}

\begin{proof}
Since $v$ is non-admissible Lemma \ref{admlem1} implies there exist critical sets $X$ on $y,z$, $Y$ on $x,z$ and $Z$ on $x,y$. From 
Lemma \ref{admlem2} we deduce that $X\cup Y$ is critical and hence $X\cup Y=V\setminus v$, since $x,y,z \in X\cup Y$.
\end{proof}

The next lemma, an analogue of \cite[Lemma $3.3$]{B&J} gives a crucial structural result about 3-connected circuits with no non-trivial 3-edge-cutsets containing non-admissible nodes. Figure \ref{nonadleaf} illustrates this; see also the first graph in Figure \ref{fig:example} for an example of a non-admissible series node.

\vspace{-.5cm}

\begin{center}
\begin{figure}[ht]
\centering
\includegraphics[width=5cm]{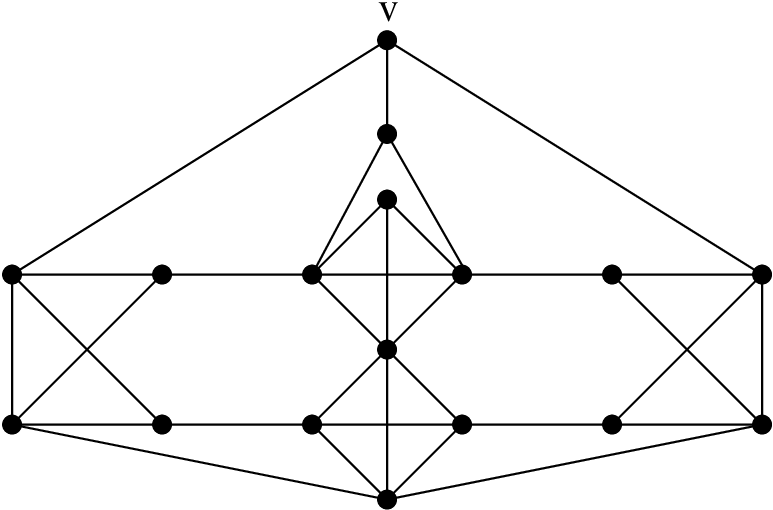}
\caption{A $3$-connected circuit with no non-trivial $3$-edge-cutsets. $v$ is a non-admissible leaf node.}
\label{nonadleaf}
\end{figure}
\end{center}

\vspace{-1cm}

\begin{lem}\label{lem3.312}
Let $G=(V,E)$ be a $3$-connected circuit with no non-trivial $3$-edge-cutsets. Let $v \in V$ be a node with $N(v)=\{x,y,z\}$, 
$d_G(z) \geq 4$ and suppose no pair of neighbours of $v$ defines an edge. Let $X$ be a $v$-critical set on $x,y$. Furthermore suppose that 
either
\begin{enumerate}
\item there is a non-admissible series node $u \in V \setminus X \setminus v$ with no edges between its neighbours, precisely one neighbour
$w$ in $X$ and $w$ is a node, or
\item there is a non-admissible leaf node $t \in V \setminus X \setminus v$ with no edges between its neighbours.
\end{enumerate}
Then there is a node-critical set $X'$ in $G$ with $|X'| >|X|$ and $(X \cap V_3^*) \subseteq (X' \cap V_3^*)$.
\end{lem}

\begin{proof}
First let $u \in V \setminus X \setminus v$ be a non-admissible series node with $N(u)=\{w,p,n\}$
and $d_G(w)=3$. We may assume $d_G(p)=3$ and $d_G(n)\geq4$.
Since $u$ is non-admissible and $wp \notin E$ there exists a $u$-critical set $Y$ on $w$ and $p$ by Lemma \ref{admlem1}. By Lemma
\ref{nok4forest} $G[V_3^*]$ contains no cycles. Note $|Y| \geq 5$ since $p,w$ are not in a copy of $K_4$.
Now $X\cap Y$ contains $w$ so $X' :=X \cup Y \subseteq V \setminus u \setminus v$ is node-critical on $u$ by Lemma 
\ref{admlem2}. Also $p \notin X$ and $d_G(n)\geq 4$ so $|X'| \geq |X|$ and $(X \cap V_3^*)\subseteq (X'\cap V_3^*)$.

For the second part of the lemma let $t$ be a non-admissible leaf node. Lemma \ref{admlem8} implies that there exist two $t$-critical 
sets $Y_1$ and $Y_2$ with $Y_1 \cup Y_2=V \setminus t$ and if $t$ has a neighbour $r$ which
is a node then we can also assume $r \in Y_1 \cap Y_2$. Note that $Y_1$ and $Y_2$ are node-critical and 
$|Y_1|, |Y_2| \geq 5$.

Now $x,y \in Y_1 \cup Y_2$ and Lemma \ref{admlem2} implies that $d(Y_1 \setminus Y_2,Y_2 \setminus Y_1)=0$.
Since also $Y_1 \cup Y_2= V \setminus t$ and 
$t \notin X$ we know that $|X \cap Y_1|\geq 1$ or $|X \cap Y_2|\geq 1$. Without loss of generality assume
$|X \cap Y_1|\geq 1$.
$d(t,X)=3$ implies $i(X\cup t)>2|X\cup t|-2$ so $d(t,X) \leq 2$. Moreover $d(t,X) \leq 1$ as if it 
were equal to $2$ then $X\cup t$ is critical and the result follows.

Now $|N(t)\cap X| \leq 1$. 
First suppose $|N(t)\cap X|=0$. Lemma \ref{admlem2} implies that $X \cup Y_1$ is $t$-critical. Thus choosing
$X'=X \cup Y_1$ completes the proof in this case.
Now suppose $N(t)\cap X = \{s\}$. If $s \in Y_1$ then $N(t) \setminus (X \cup Y_1) \neq \emptyset$ (as $N(t)\nsubseteq Y_1$) and hence
$X'=X \cup Y_1$ is node-critical and we are done.
If $d_G(s)=3$ then $s \in Y_1 \cap Y_2$ so we may assume $d_G(s)\geq 4$ and $s \notin Y_1$. Since 
$Y_1 \cup Y_2=V\setminus t$ this gives $s \in Y_2$. $|X \cap Y_2|\geq 1$ so choose $X'=X \cup Y_2$ to 
complete the proof.
\end{proof}

Similarly to \cite[Lemmas $3.5$ and $3.6$]{B&J} we have the following two lemmas.

\begin{lem}\label{admlem3}
Let $G$ be a $3$-connected circuit with no non-trivial $3$-edge-cutsets and $|V| \geq 6$. Let 
$\X=\{X \subset V:X \mbox{ is a node-critical set in } G\}$. If $\X=\emptyset$ then $G$ has two admissible nodes.
\end{lem}

\begin{proof}
By Lemma \ref{nok4forest} $|V_3^*|\geq 2$.
Since $\X=\emptyset$ the result follows from Lemmas \ref{admlem1} and \ref{star}. 
\end{proof}


\begin{lem}\label{admlem99}
Let $G$ be a $3$-connected circuit with no non-trivial $3$-edge-cutsets and $|V| \geq 6$. Suppose $v$ is an admissible node. Let 
$\Y=\{Y \subset V: v\in Y,\mbox{ } Y \mbox{ is a node-critical set in } G\}$. If $\Y=\emptyset$ then $G$ has two admissible nodes.
\end{lem}

\begin{proof}
By Lemma \ref{nok4forest} $|V_3^*|\geq 2$.
Let $w\neq v$ be a leaf in $G[V_3^*]$ and suppose $w$ is non-admissible. Either this contradict Lemma \ref{star}
or by Lemma \ref{admlem8} there exist node-critical sets $X,Y$ with $X\cup Y=V\setminus w$, contradicting $\Y=\emptyset$.
\end{proof}

\subsection{Proof of Theorem \ref{recursivethm2}}

We are now ready to prove that any sufficiently connected circuit contains an admissible vertex.

\begin{proof}[Proof of Theorem \ref{recursivethm2}]
By Lemma \ref{nok4forest} $G[V_3^*]$ is a forest and $|V_3^*|\geq 2$.
By Lemma \ref{star} we need consider only the case when there are no edges between the neighbours of
every $a\in V_3^*$. 

Let $\X=\{X\subset V: X \mbox{ is a node-critical set in } G\}$. If $\X=\emptyset$ we 
are done by Lemma \ref{admlem3}. Otherwise let $X \in \X$ be maximal.
Choose $t \in N(v)$ such that $X$ is $v$-critical with $d_G(t)\geq 4$ and $t \notin X$. $X\cup v$ is critical and
$|V\setminus X\setminus v|\geq 2$, otherwise $i(X\cup v\cup t)>2|X\cup v\cup t|-1$.
By Lemma \ref{admlem5} $V\setminus X\setminus v$ contains a node. 

Let $X=X_n$ and let $X_1,\dots,X_{n-1}$ be critical sets in 
$G$ not contained in $X$ such that every copy of $K_4$ is induced by some 
$X_i$ and every $X_i$ induces a copy of $K_4$. Then there are two cases. If $t \notin X_i$ for all $i$ then 
$|Y|= |V \setminus \bigcup_{i=1}^n X_i|\geq 2$ so Proposition \ref{prop:k4fix} part $1$ implies there is a vertex not in $X\cup v$ which is a node 
not in a copy of $K_4$. Secondly if $t \in X_i$ for some $i$ then $|X\cap X_i|\leq 1$ otherwise $i(X\cup X_i)>2|X\cup X_i|-2$.
Moreover if $X\cap X_i = a$ then $X\cap X_i$ is critical so $d(X,X_i)=0$ and $X\cup X_i\cup v = V$ implying $a,v$ is a cut-pair for $G$.
Hence $|X\cap X_i|=0$ and $\bigcup_{i=1}^n G[X_i]$ is disconnected so Proposition \ref{prop:k4fix} part $2$ implies there is a vertex not in 
$X\cup v$ which is a node not in a copy of $K_4$.

Let $W^*:=V_3^*\cap (V\setminus X\setminus v)$.
$G[W^*]$ is a subforest of $G[V_3^*]$ on the vertex set $W^*$. By the preceding paragraph $|W^*|\geq 1$ so $W$ contains a leaf $u$. 
Each vertex $z \in V\setminus X\setminus v\setminus t$ has at most one neighbour in $X$; otherwise $X\cup z$ is node-critical, 
contradicting the maximality of $|X|$. Therefore $u$ is not a branching node of $G$.

Now if $u$ is a leaf node then Lemma \ref{lem3.312} part $2$ and the maximality of $|X|$ imply that $u$ is an admissible node. If 
$u$ is a series node in $G$ then, since $u$ has at most one neighbour in $X$
and since $u$ is a leaf in $G[W^*]$, it follows that it has precisely one neighbour $y$ in $X$ and $y$ is a node.
Thus Lemma \ref{lem3.312} part $1$ and the maximality of $|X|$ imply that $u$ is an admissible node.

Finally let $\Y=\{Y \subset V: u\in Y,\mbox{ } Y \mbox{ is a node-critical set in } G\}$. If $\Y=\emptyset$ the result follows from
Lemma \ref{admlem99}. Otherwise let $Y\in \Y$ be maximal, and argue similarly to the proof for $X\in \X$ to complete the proof. 
\end{proof}

\section{Joining Circuits}
\label{sumsec}

By Lemma \ref{admlem9} and Theorem \ref{recursivethm1} it remains to consider the generation of circuits with cutpairs or with non-trivial $3$-edge cutsets. 

\subsection{Circuits containing cut-pairs}

We start by considering graphs that
are not $3$-connected. Let $K_n(a_1,\dots,\break a_n)$ denote the complete graph with vertex set $\{a_1,\dots,a_n\}$.
Let $G=(V,E)$ be a circuit with a cutpair $a,b$ and a bipartition $A,B$ of $V\setminus \{a,b\}$. Since $f(G)=1$ and $f(H)\geq 2$ 
for all subgraphs there are two options: $ab \in E$ and $f(G[A\cup\{a,b\}])=f(G[B\cup \{a,b\}])=2$ or 
$ab \notin E$ and $3=f(G[A\cup\{a,b\}])<f(G[B\cup \{a,b\}])=2$. This leads us to the $1$- and $2$-join operations. To refresh the readers memory we define the inverse operations.

Let $G$ be as above and suppose
$f(G[A\cup\{a,b\}])< f(G[B\cup\{a,b\}])$. A \emph{$1$-separation} over the cutpair $a,b$ forms disjoint graphs 
$G[A\cup \{a,b\}]\cup ab$ and $G[B\cup \{a,b\}]\cup K_4(a,b,c,d)$ where $c,d \notin B\cup\{a,b\}$.
Also let $G=(V,E)$ be a circuit with a cutpair $a,b$ with a bipartition $A,B$ of $V\setminus \{a,b\}$ such that 
$f(G[A\cup\{a,b\}])= f(G[B\cup\{a,b\}])$. A \emph{$2$-separation} over the cutpair $a,b$ forms disjoint graphs 
$G[A\cup \{a,b\}]\cup K_4(a,b,c,d)$ and $G[B\cup \{a,b\}]\cup K_4(a,b,c,d)$ where $c,d \notin A\cup\{a,b\}$ or $B\cup\{a,b\}$.

\begin{lem}\label{1sum}
Let $G_1=(V_1,E_1)$, $G_2=(V_2,E_2)$ be graphs such that $G_1$ contains an edge $a_1b_1$ and $G_2$ contains a two vertex cut
$a_2,b_2$ within\break $K_4(a_2,b_2,c_2,d_2)$. Then the $1$-join $G_1 \oplus_1 G_2=G=(V,E)$ (merging $a_1=a_2$ into $a$ and $b_1=b_2$ into 
$b$) is a circuit if and only if $G_1$ and $G_2$ are circuits. 
\end{lem}

\begin{proof}
We have $V=(V_1\setminus\{a_1,b_1\}) \cup (V_2\setminus\{a_2,b_2,c_2,d_2\}) \cup \{a,b\}$ so 
\begin{eqnarray*} 
|E|=|E_1|-1+|E_2|-6&=&2|V_1|-1+2|V_2|-1-7\\&=&2(|V_1|+|V_2|-4)-1=2|V|-1.
\end{eqnarray*} 

Let $X\subset V$. Let $X_i=(V_i\cap X)\cup (\{a,b\}\cap X)$ and let $X_i'=(V_i\cap X)\cup (\{a_i,b_i\}\cap X)$.
If $X$ contains both $a$ and $b$ then
\begin{eqnarray*} 
i_G(X) & = & i_{G_1}(X_1')+i_{G_2}(X_2')-2\leq 2|X_1'|-1+ 2|X_2'|-2-2\\ &=& 2(|X_1'|+|X_2'|)-5= 2|X|-1.
\end{eqnarray*}
where equality holds if and only if $X=V$. Similarly, if $X$ contains at most one of $a$ and $b$ then $i_G(X)\leq 2|X|-2$.

Conversely, suppose $G_1$ is not a circuit. Since $|E_1|=2|V_1|-1$ there exists $X$ properly contained
in $A\cup \{a,b\}$ with $i_{G_1}(X)=2|X|-1$. $X$ contains $a,b$ otherwise $X\subset V$. We have 
\begin{eqnarray*} 
i_G(X\cup B\cup \{a,b\}) & = & 2|X|-2+2|B\cup \{a,b\}|-3\\ &=&
2(|X\setminus \{a,b\}|+|B\cup \{a,b\}|-2)-1,
\end{eqnarray*}
a contradiction.

Now suppose $G_2$ is not a circuit. Since $|E_2|=2|V_2|-1$ there exists $X$ properly contained
in $B\cup \{a,b,c,d\}$ with $i_{G_2}(X)=2|X|-1$. 
$X$ contains $c,d$ otherwise $X$ is a subset of $V$ and thus $X$ contains $a,b$. We 
have 
\begin{eqnarray*} 
i_G((X\setminus \{c,d\})\cup A\cup \{a,b\}) & = & 2|X\setminus\{c,d\}|-2+2|A\cup \{a,b\}|-2-1\\ &=&
2(|X\setminus \{c,d\}|+|A\cup \{a,b\}|-2)-1,
\end{eqnarray*}
a contradiction.
\end{proof}

\begin{lem}\label{2sum}
Let $G_1=(V_1,E_1),G_2=(V_2,E_2)$ be graphs such that $G_i$ contains a two vertex cut
$a_i,b_i$ within $K_4(a_i,b_i,c_i,d_i)$. Then the $2$-join $G_1 \oplus_2 G_2=(V,E)$ (merging $a_1=a_2$ into $a$ and $b_1=b_2$ into 
$b$) is a circuit if and only if $G_1$ and $G_2$ are circuits. 
\end{lem}

\begin{proof}
We have $V=(V_1\setminus\{a_1,b_1,c_1,d_1\}) \cup (V_2\setminus\{a_2,b_2,c_2,d_2\}) \cup \{a,b\}$ so
 \begin{eqnarray*} 
|E|=|E_1|-6+|E_2|-6+1&=&2|V_1|-1+2|V_2|-1-11\\&=&2(|V_1|+|V_2|-6)-1=2|V|-1.
\end{eqnarray*} 

Let $X\subset V$. Let $X_i=(V_i\cap X)\cup (\{a,b\}\cap X)$ and let $X_i'=(V_i\cap X)\cup (\{a_i,b_i\}\cap X)$.
If $X$ contains both $a$ and $b$ then
\begin{eqnarray*} 
i_G(X) & = & i_{G_1}(X_1')+i_{G_2}(X_2')-1\leq 2|X_1'|-2+ 2|X_2'|-2-1\\ &=& 2(|X_1'|+|X_2'|-2)-1= 2|X|-1.
\end{eqnarray*}
where equality holds if and only if $X=V$. Similarly, if $X$ contains at most one of $a$ and $b$ then $i_G(X)\leq 2|X|-2$.

For the converse, by symmetry, it is enough to show that $G_1$ is a circuit.

Suppose $G_1$ is not a circuit. Since $|E_1|=2|V_1|-1$ there exists $X$ properly contained
in $A\cup \{a,b,c,d\}$ with $i_{G_1}(X)=2|X|-1$. $X$ contains $c,d$ otherwise $X$ is a subgraph of $G$ and thus $X$ contains $a$ and $b$. We 
have 
\begin{eqnarray*} 
i_G((X\setminus \{c,d\})\cup (B\cup \{a,b\})) & = & 2|X\setminus\{c,d\}|-2+2|B\cup \{a,b\}|-2-1\\ &=&
2(|X\setminus \{a,b\}|+|B\cup \{a,b\}|-2)-1,
\end{eqnarray*}
a contradiction.
\end{proof}

\subsection{Circuits with 3-edge-cutsets}

We also require the $3$-join operation. Let $G=(V,E)$ be a circuit with a non-trivial $3$-edge-cutset $a_1a_2,b_1b_2,c_1c_2$ with a bipartition $A,B$ of $V$ such that 
$f(G[A])= f(G[B])$. A \emph{$3$-separation} over the cutset $a_1a_2,b_1b_2,c_1c_2$ forms disjoint graphs 
$G[A]\cup v_1 \cup \{a_1v_1,b_1v_1,c_1v_1\}$ and $G[B]\cup v_2\cup\break \{a_2v_2,b_2v_2,c_2v_2\}$. 

\begin{lem}\label{3sum}
Let $G_1=(V_1,E_1),G_2=(V_2,E_2)$ be graphs. Then the $3$-join 
$G=G_1 \oplus_3 G_2=(V,E)$ (deleting $v_i\in V_i$ with $d_{G_i}(v_i)=3$ and $N(v_i)=\{a_i,b_i,c_i\}$ for $i=1,2$ and adding 
$a_1a_2,b_1b_2,c_1c_2$) is a circuit if and only if $G_1$ and $G_2$ are circuits. 
\end{lem}

\begin{proof}
We have $V=(V_1\setminus v_1) \cup (V_2\setminus v_2)$ so
 \begin{eqnarray*} 
|E|=|E_1|-3+|E_2|-3+3&=&2|V_1|-1+2|V_2|-1-3\\&=&2(|V_1|+|V_2|-2)-1=2|V|-1.
\end{eqnarray*} 

Let $X\subset V$. Let $X_i=(V_i\cap X)$.
$X$ contains at least one of $a_i,b_i,c_i$, otherwise $X \subset X_i$ and so $i(X)\leq 2|X|-2$. Let $0\leq t\leq 3$ be the number of 
edges in the subgraph induced by $X$ from the set $\{a_1a_2,b_1b_2,c_1c_2\}$. Then
\begin{eqnarray*} 
i_G(X) & = & i_{G_1}(X_1)+i_{G_2}(X_2)+t\\ &\leq& 2|X_1|-2+ 2|X_2|-2+t\\ &\leq& 2|X|-1.
\end{eqnarray*}
where equality holds if and only if $X=V$; otherwise for some $i$, $X_i \subsetneq V_i$, $i(X_i)=2|X_i|-2$ and $X_i$ contains $a_i,b_i,c_i$
so adding back $v_i$ contradicts $G_i$ being a circuit.

For the converse, clearly $f(G[A])= f(G[B])=2$.
By symmetry it is enough to show that $G_1$ is a circuit.

Suppose $G_2$ is not a circuit. Since $|E_1|=2|V_2)|-1$ there exists $X$ properly contained
in $A\cup v_1$ with $i_{G_1}(X)=2|X|-1$. $X$ contains $v_1$, otherwise $X$ is a subgraph of $G$, and thus contains $a_1,b_1,c_1$. We 
have 
\begin{eqnarray*} 
i_G((X\setminus v_1) \cup B) & = & 2|X\setminus v_1|-2+2|B|-2+3\\ &=&
2(|X\setminus v_1|+|B|)-1,
\end{eqnarray*}
a contradiction.
\end{proof}

\section{A Recursive Construction of Circuits}
\label{sec:recursion}

It remains to deal with the case when every cutpair $a,b$ in $G$ with associated bipartition $A,B$ is such that at least one of the 
subgraphs induced by $A\cup\{a,b\}$ and $B\cup\{a,b\}$ is isomorphic to $K_4$. Here the $2$-separation move results in a copy of
$G$ and a copy of $K_4 \sqcup K_4$. However we do not need a new recursive move to deal with this case. Consider a graph $G$
with $n$ cutpairs and each cutpair $a_i,b_i$ with bipartition $A_i,B_i$ leaves $G[A_i\cup\{a_i,b_i\}]$ isomorphic to 
$K_4(a_i,b_i,c_i,d_i)$. Now delete each $c_i,d_i$ and all incident edges and add a second copy of each edge $a_ib_i$. We denote the
resulting multigraph as $G^-=(V^-,E^-)$. $G^-$ is a $3$-connected \emph{multicircuit}, see Figure \ref{usemulti}. None of the $a_i$ or 
$b_i$ are nodes; if $d_G(a_i)=3$ then $N(a_i)=\{b_i,x\}$ for some $x$ but then $b_i,x$ is a cutpair for $G^-$ and hence for $G$. 
Thus every node in $G^-$ has $3$ distinct neighbours. 

There is a node in a multicircuit in which an inverse Henneberg $2$ move
results in a multicircuit by Frank and Szeg\"{o} \cite[Theorem $1.10$]{F&S2}.
However we need the following stronger result which follows by the same proof as Theorem \ref{recursivethm2},  noting that the simplicity assumption did not provide a simplification. By admissible here we mean that there is an inverse Henneberg $2$ move on a node that results in a circuit and that the new edge does not create a double edge.

\begin{center}
\begin{figure}[ht]
\centering
\includegraphics[width=6.7cm]{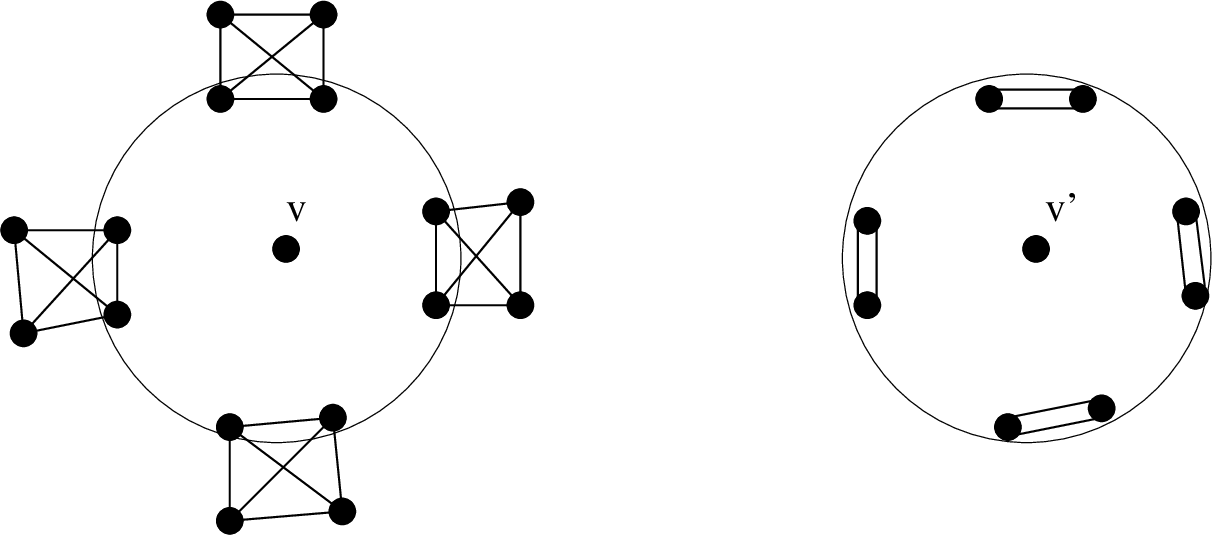}
\caption{For every $2$-vertex cutset with one component a copy of $K_4$, replace each copy with a double edge. We show that if $v'$ is an admissible node then so is $v$.}
\label{usemulti}
\end{figure}
\end{center}

\vspace{-1.1cm}

\begin{prop}\label{3conmulti}
Let $G=(V,E)$ be a multigraph with $|V|\geq 6$. Let $G$ be a $3$-connected multicircuit with no non-trivial $3$-edge-cutsets in which 
every node has $3$ distinct neighbours. Then $G$ contains an admissible node. 
\end{prop}

\subsection{Proof of Theorem \ref{recursivethm1}}

We are now ready to prove our main result.

\begin{proof}[Proof of Theorem \ref{recursivethm1}]
By Lemmas \ref{hen2move}, \ref{1sum}, \ref{2sum} and \ref{3sum} a connected graph built up recursively from disjoint copies of base 
graphs by $1$-joins, $2$-joins, $3$-joins and Henneberg $2$ moves is a circuit.

Conversely we use induction on $|V|$. Since $K_5\setminus e$ is the unique circuit on at most $5$ vertices, by Theorem \ref{recursivethm2},
we may apply an inverse Henneberg $2$ move whenever $G$ is $3$-connected with no non-trivial $3$-edge cutsets. 
If $G$ is $3$-connected with a non-trivial $3$-edge-cutset then, by Lemma \ref{3sum} we may apply a $3$-separation to $G$ resulting
in smaller circuits.

If $G$ is not $3$-connected then there is a cutpair. Choose a cutpair $a,b$. If $ab \notin E$ then by Lemma \ref{1sum} we can apply a 
$1$-separation in such a way that the resulting graphs are circuits. 
Suppose then for every cutpair $a,b$, $ab \in E$ and suppose there is a choice of $a,b$ such that $G[A\cup\{a,b\}]$ and $G[B\cup\{a,b\}]$ 
are not isomorphic to $K_4$. Then by Lemma \ref{2sum} we can apply a $2$-separation in such a way that the resulting graphs are 
circuits. 

Now if every minimal choice of cutpair results in $G[A\cup\{a,b\}]\cong\break K_4(x_i,y_i,z_i,w_i)$ where $x_i,y_i$ is the 
cutpair and the corresponding multigraph $G^+$, as above, has $|V^+|\geq 6$ then 
the result follows from Proposition \ref{3conmulti}.

It remains to check the cases when $|V^+|\leq 5$. If $|V^+|=2$ then $G\cong K_4 \sqcup K_4$. If $|V^+|=3$ then $G\cong K_4\veebar K_4$.
If $|V^+|=4$ or $|V^+|=5$ there are a small number of cases that are each easy to check (there is 
an admissible node or a separation to smaller circuits).
\end{proof}

\section{Connected Matroids and Rigid Frameworks}
\label{sec:rigidity}

In the remainder of the paper we consider potential applications of our results to frameworks on surfaces.

\subsection{Rigidity on the cylinder}

A \emph{framework} $(G,p)$ on the cylinder $S^1\times \bR$ in $\bR^3$ is the combination of a graph $G$ and a map $p:V\rightarrow S^1\times \bR$. We will 
focus only on when such frameworks are generic: there are no algebraic dependencies among the coordinates of the framework points 
that are not required by $\M$. The \emph{cylinder rigidity matrix} $R_{S^1\times \bR}(G,p)$ is the $(|E|+|V|) \times 3|V|$ matrix where the first 
$|E|$ rows correspond to the edges and the entries in the row for edge $uv$ are $0$ except in the column triples corresponding to
$u$ and $v$ where the entries are $p(u)-p(v)$ and $p(v)-p(u)$ respectively. The final $|V|$ rows correspond to the vertices and the 
entries in the row for vertex $i$ are zero except in the column triple corresponding to $i$ where the entry is $N(p(i))$, the surface 
normal to the point $p(i)$. A framework $(G,p)$ on $S^1\times \bR$ is \emph{generic} if the only polynomial equations satisfied by the coordinates of $p$ are those that define $S^1\times \bR$.
Let $\R_{S^1\times \bR}$ denote the \emph{cylinder rigidity matroid}, that is the linear matroid induced by linear independence in the rows
of $R_{S^1\times \bR}(G,p)$ for generic $p$. A framework is \emph{infinitesimally rigid} if its edge set has maximal rank in $\R_{S^1\times \bR}$. 

More detailed definitions may be found in \cite{NOP}, see also \cite{GSS} for a detailed study of rigidity matroids.

\begin{thm}[\cite{NOP}]\label{t:lamancylinder}
Let $G=(V,E)$ be a graph with $|V|\geq 4$ and let $(G,p)$ be a generic framework in $3$-dimensions constrained to $S^1\times \bR$. 
Then the matroids $\R_{S^1\times \bR}$ and $M^*(2,2)$ are isomorphic.
\end{thm}

Similarly if $\R_2$ denotes the rigidity matroid for generic frameworks in $\bR^2$, then Laman's theorem \cite{Lam} states $\R_2\cong M(2,3)$.
We will need the following corollary to Theorem \ref{t:lamancylinder}. A \emph{redundantly rigid} framework $(G,p)$ on $S^1\times \bR$ is a framework
such that after deleting any single edge from $G$ the rigidity matroid still has maximal rank.

\begin{cor}\label{corredundantcircuits}
 Let $G=(V,E)$ and let $p$ be generic. Then $(G,p)$ is redundantly rigid on $S^1\times \bR$ if and 
only if $(G,p)$ is infinitesimally rigid on $S^1\times \bR$ and every edge of $G$ belongs to a $\R_{S^1\times \bR}$-circuit.
\end{cor}

\begin{rem}\label{ordermoves}{\rm
By Theorem \ref{t:lamancylinder} a generic framework $(G,p)$ on $S^1\times \bR$ is
rigid if and only if $G$ contains a spanning $(2,2)$-tight subgraph.
However as $K_{3,6}$ illustrates (see also \cite[Figure $6$]{J&J} for the plane case) extending Theorem \ref{recursivethm1} from circuits to 2-connected redundantly rigid graphs is non-trivial.
For example $K_{3,6}$ is not a circuit so one of the 
operations must be an edge addition.
The last move must be a Henneberg $2$ move since $K_{3,6}$ is $3$-connected with no non-trivial $3$-edge cutsets and minimal in the 
sense that removing any edge results in a graph $G=(V,E)$ with $|E|=2|V|-1$ that is not a circuit.
}\end{rem}

\subsection{$\R_{S^1\times \bR}$-connected Graphs}
\label{rmconsec}

Following \cite{J&J}, for $\R_{S^1\times \bR}=(E,I)$, define a relation on $E$ by saying $e,f \in E$ are related if $e=f$owing \cite{J&J}, for $\R_{S^1\times \bR}=(E,I)$, define a relation on $E$ by saying $e,f \in E$ are related if $e=f$
or if there is a $\R_{S^1\times \bR}$-circuit $C$ with $e,f \in C$. 
We abuse notation slightly by referring to $C$ as both the circuit in $\R_{S^1\times \bR}$ and the graph induced by the circuit, i.e.
the $\R_{S^1\times \bR}$-circuit.
This is an equivalence relation and the equivalence classes are the components
of $\R_{S^1\times \bR}$. If $\R_{S^1\times \bR}$ has at least two elements and only one component then it is \emph{$\R_{S^1\times \bR}$-connected}. $G$ is \emph{$\R_{S^1\times \bR}$-connected} if
$\R_{S^1\times \bR}$ is connected. The $\R_{S^1\times \bR}$-components of $G$ are the subgraphs of $G$ induced by the components of $\R_{S^1\times \bR}$.

Since bases in $M^*(2,2)$ can contain cut-vertices while circuits cannot, to link redundantly rigid frameworks and $\R_{S^1\times \bR}$-connected graphs requires 2-connectivity.

\begin{thm}\label{Mconnectedthm*}
A graph $G$ is $2$-connected with a redundantly rigid realisation on $S^1\times \bR$ if and only if $G$ is 
$\R_{S^1\times \bR}$-connected.
\end{thm}

\begin{proof}
Suppose $G$ is $\R_{S^1\times \bR}$-connected. $G$ is infinitesimally rigid since there is only one $\R_{S^1\times \bR}$-connected component. $\R_{S^1\times \bR}$ is connected so every 
edge is in a $\R_{S^1\times \bR}$-circuit. Thus $G$ has a redundantly rigid realisation by Corollary \ref{corredundantcircuits}.
Also Lemma \ref{admlem9} implies $G$ is $2$-connected.

Conversely let $X$ be the set of $\R_{S^1\times \bR}$-connected components of $G$ and $\theta(X)$ the set of vertices of $G$ 
belonging to two distinct elements of $X$. Let $d_X(v)$ denote the number of elements of $X$ containing $v$. Let $r(G)$ denote the rank
of the rigidity matroid $\R_{S^1\times \bR}(G,p)$.
Then
\[ 2|V|-2=r(G)=\sum_{H \in X}r(H)=\sum_{H \in X}(2|V(H)|-2) \]
and
\[ |V|=\sum_{H \in X}|V(H)|-\sum_{v \in \theta(X)}(d_{X}(v)-1). \]
This implies that $\sum_{v \in \theta(X)}d_X(v) < 2|X|$ so there exists $H \in X$ with 
$|V(H) \cap \theta(X)| \leq 1$.
\end{proof}

\subsection{Global Rigidity}
\label{concludingremarks}

\begin{defn}
A framework $(G,p)$ on $S^1\times \bR$ is \emph{globally rigid} if every framework $(G,q)$ which satisfies the (Euclidean $3$-space) distance constraint equations $|p_{i}-p_{j}|=|q_{i}-q_{j}|$, for each edge $ij$ where $p_{i},p_{j},q_{i},q_{j}$ are points on $S^1\times \bR$ also satisfies $|p_{i}-p_{j}|=|q_{i}-q_{j}|$ for every pair of vertices $i,j$ of $G$.
\end{defn}

We now recall the celebrated characterisation of generic global rigidity in the plane. This is due, in its various parts, to Connelly \cite{Con}, 
Hendrickson \cite{Hen} and Jackson and Jord\'{a}n \cite{J&J}.
Giving a full $3$-dimensional combinatorial characterisation remains a hard open problem.

\begin{thm}
Let $G=(V,E)$ with $|V|\geq 4$ and let $p$ be generic. Then the following
 are equivalent:
\begin{enumerate}
\item[$(1)$] $(G,p)$ is globally rigid in $\bR^2$, 
\item[$(2)$] $G$ is $3$-connected and $(G,p)$ is redundantly rigid in the plane,
\item[$(3)$] $G$ can be formed from disjoint copies of $K_4$ by Henneberg $2$ 
moves and edge additions,
\item[$(4)$] $G$ is $3$-connected and $\R_2$-connected.
\end{enumerate}
\end{thm}

The analysis in this paper leads us to make the following conjecture.

\begin{con}\label{cylinderglobal}
Let $G=(V,E)$ with $|V|\geq 5$ and let $p$ be generic for $S^1\times \bR$. The following
 are equivalent:
\begin{enumerate}
\item[$(1)$] $(G,p)$ is globally rigid on $S^1\times \bR$, 
\item[$(2)$] $G$ is $2$-connected and $(G,p)$ is redundantly rigid on $S^1\times \bR$,
\item[$(3)$] $G$ can be formed from disjoint copies of $K_5\setminus e, K_4\sqcup K_4$ and $K_4 \veebar K_4$ by Henneberg $2$ 
moves, $1$-joins, $2$-joins, $3$-joins and edge additions,
\item[$(4)$] $G$ is $\R_{S^1\times \bR}$-connected.
\end{enumerate}
\end{con}

For $|V|\leq 4$, $(G,p)$ is globally rigid on $S^1\times \bR$ if and only if $G$ is a complete graph.
Following the submission of this paper, $(1) \Rightarrow (2)$ has been confirmed in \cite{JMN}.
Theorem \ref{Mconnectedthm*} shows the equivalence of $(2)$ and $(4)$.

\section{Concluding Remarks}

Our conjectured characterisation would provide a sufficient condition for global rigidity on the cylinder that fails 
somewhat trivially in the plane.
Let $G$ contain a spanning subgraph $H$ which is a $\R_{S^1\times \bR}$-circuit and let $p$ be generic for $S^1\times \bR$. Then Conjecture \ref{cylinderglobal} implies that
$(G,p)$ is globally rigid on $S^1\times \bR$.
Remark \ref{ordermoves} illustrates why this does not characterise globally rigid frameworks on the cylinder.

The special case in which $G$ has the minimum possible number of edges $2|V|-1$ corresponding to \cite[Theorem $6.1$]{B&J} conjectures that the generically globally rigid graphs on the cylinder are exactly the $\R_{S^1\times \bR}$-circuits. 
To prove the minimal case it remains to show that 
the Henneberg $2$ and $i$-join moves preserve global rigidity.

The remaining combinatorial difficulty in Conjecture \ref{cylinderglobal} is in showing that every $\R_{S^1\times \bR}$-connected graph can be generated using only the construction 
moves in Theorem \ref{recursivethm1}. In the case of the plane this was done by Jackson and Jord\'{a}n \cite{J&J} who used the concept of an 
ear decomposition in a $\R_2$-connected graph. Such a theorem would complete the equivalence 
of $(2),(3)$ and $(4)$.

Conjecture \ref{cylinderglobal} would lead to an efficient algorithm for checking global rigidity. $2$-connectedness can be checked in 
linear time \cite{H&T} and redundant rigidity, via the pebble game \cite{H&J}, \cite{L&S}, can be checked in $O(|V|^2)$ time.

Finally we note that Theorems \ref{recursivethm2} and \ref{recursivethm1} do not easily extend to the case of circuits in $M^*(2,1)$. A higher level of connectivity will be required to guarantee an admissible node when a node even exists. Moreover circuits in $M^*(2,1)$ may contain cut-vertices and more elaborate $i$-join operations may be required. A characterisation of circuits in $M^*(2,1)$ would be a step towards proving the analogue of Conjecture \ref{cylinderglobal} for frameworks on a surface of revolution \cite{NOP2}, such as a cone \cite{JMN}.

\subsection{Acknowledgements}

We thank Bill Jackson for many productive discussions on this topic. Much of Section \ref{ansec} was inspired by the paper 
of Berg and Jord\'{a}n \cite{B&J}. The majority of this research was carried out at the Fields Institute, Toronto.


\begin{thebibliography}{00}
\bibitem{B&J} A. Berg and T. Jord\'{a}n, A Proof of Connelly's Conjecture on $3$-connected Circuits of the Rigidity Matroid, \emph{J. Combinatorial Theory B}, 88, (2003), 77-97.
\bibitem{Con} R. Connelly, Generic Global Rigidity, \emph{Discrete Comput. Geom.}, 33 (2005), 549-563.
\bibitem{F&Z} Z. Fekete and L. Szeg\"{o}, A Note on $[k,l]$-sparse Graphs, in Graph Theory in Paris, Edited by 
A. Bondy et al, (2007), 169-177.
\bibitem{F&S2} A. Frank and L. Szeg\"{o}, Constructive Characterizations for Packing and Covering with Trees, \emph{Discrete Appl. Math.},
131, 347-371, (2003).
\bibitem{GSS} J. Graver, B. Servatius and H. Servatius, Combinatorial Rigidity, AMS Graduate Studies in mathematics, 2 (1993).
\bibitem{Hdk} B. Hendrickson, Conditions for Unique Graph Realisations, \emph{SIAM J. Comput.} 21, 1, (1992), 65-84.
\bibitem{H&J} B. Hendrickson and D. Jacobs, An Algorithm for two-dimensional Rigidity Percolation: the Pebble
Game, \emph{J. Computational Physics}, 137, (1997), 346-365.
\bibitem{Hen} L. Henneberg, Die graphische Statik der starren Systeme, Leipzig (1911).
\bibitem{H&T} J. Hopcroft and R. Tarjan, Efficient Algorithms for Graph manipulation, \emph{Communications of the ACM}, 16, 6, (1973) 372-378.
\bibitem{J&J} B. Jackson and T. Jord\'{a}n, Connected Rigidity Matroids and Unique Realisations of Graphs, \emph{J. Combinatorial Theory B}, 94, (2005), 1-29.
\bibitem{JMN} B. Jackson, T. McCourt and A. Nixon, Necessary conditions for the generic global rigidity of frameworks on surfaces, http://arxiv.org/abs/1306.2346, (2013).
\bibitem{Lam} G. Laman, On Graphs and the Rigidity of Plane Skeletal Structures, \emph{J. Engrg. Math.}, 4 (1970), 331-340.
\bibitem{L&S} A. Lee and I. Streinu, Pebble Game Algorithms and Sparse Graphs, \emph{Discrete Math.}, 308, 8, (2008), 1425-1437.
\bibitem{Oxl} J. Oxley, Matroid Theory, Oxford University Press, (1992).
\bibitem{N-W} C. St J. A. Nash-Williams, Decomposition of Finite Graphs into Forests, \emph{J. London Math. Soc.} 39, 12 (1964).
\bibitem{NOP} A. Nixon, J. Owen and S. Power, Rigidity of Frameworks Supported on Surfaces, \emph{SIAM J. Discrete Math.} Vol. 26, No. 4, pp. 1733-1757, (2012)
\bibitem{NOP2} A. Nixon, J. Owen and S. Power, A Laman theorem for frameworks on surfaces of revolution, http://arxiv.org/abs/1210.7073v2, (2013).
\bibitem{N&O} A. Nixon and J. Owen, An Inductive Construction of $(2,1)$-tight Graphs, http://arxiv.org/abs/1103.2967v2, (2011).
\bibitem{Tay2} T-S. Tay: Henneberg's Method for Bar and Body Frameworks. \emph{Structural Topology} 17, (1991), 53-58.
\bibitem{Tut} W. Tutte, On the Problem of Decomposing a Graph into $n$ Connected Factors, \emph{J. London Math. Soc.}, 142, (1961), 221-230.
\bibitem{Whi3} W. Whiteley, Matroids for Discrete Applied Geometry, in Matroid Theory, J. Bonin, J. Oxley and B. Servatius (eds.), Contemporary Mathematics 197, AMS, (1996), 171-311.
\bibitem{Whi5} W. Whiteley, The Union of Matroids and the Rigidity of Frameworks, \emph{SIAM J. Disc. Math.}, 1, 2, (1988), 237-255.
\end{thebibliography}
\end{document}